\title{$\mathbb{Z}_3$-actions on Horikawa surfaces.}

\author{Vicente Lorenzo}
\date{}
\RequirePackage{fix-cm}
\RequirePackage{amsmath}

\documentclass{article}
\usepackage{graphicx}
\usepackage[T1]{fontenc}
\usepackage[utf8]{inputenc}
\usepackage{lmodern}
\usepackage[english]{babel}
\usepackage{amssymb}
\usepackage{amsthm}
\usepackage[all]{xy}
\usepackage{enumerate}
\usepackage[none]{hyphenat}
\usepackage{etoolbox}
\usepackage{microtype}
\usepackage{lipsum}
\usepackage[pdfencoding=auto, psdextra]{hyperref}
\usepackage{tikz-cd}

\newtheorem{theorem}{Theorem}

\newtheorem{proposition}{Proposition}

\theoremstyle{remark}\newtheorem{remark}{Remark}
\theoremstyle{remark}\newtheorem{example}{Example}
\makeatletter
\renewenvironment{proof}[1][\proofname]{%
  \par\pushQED{\qed}\normalfont%
  \topsep6\p@\@plus6\p@\relax
  \trivlist\item[\hskip\labelsep\bfseries#1\@addpunct{.}]%
  \ignorespaces
}{%
  \popQED\endtrivlist\@endpefalse
}
\makeatother
\newenvironment{acknowledgements}{\textit{Acknowledgements.}}{}

\newcommand\blfootnote[1]{%
  \begingroup
  \renewcommand\thefootnote{}\footnote{#1}%
  \addtocounter{footnote}{-1}%
  \endgroup
}

\renewcommand{\thefootnote}{\textit{\alph{footnote}}}

\begin{document}

\maketitle

\begin{abstract}
Minimal algebraic surfaces of general type $X$ such that $K^2_X=2\chi(\mathcal{O}_X)-6$ are called Horikawa surfaces. In this note $\mathbb{Z}_3$-actions
on Horikawa surfaces are studied. 
The main result states that given an admissible pair $(K^2, \chi)$ such that $K^2=2\chi-6$, all the connected components of Gieseker's moduli space $\mathfrak{M}_{K^2,\chi}$ contain surfaces admitting a $\mathbb{Z}_3$-action. 
On the other hand, the examples considered allow to produce normal stable surfaces that do not admit a $\mathbb{Q}$-Gorenstein smoothing. This is illustrated by constructing  non-smoothable normal surfaces in the KSBA-compactification $\overline{\mathfrak{M}}_{K^2,\chi}$ of Gieseker's moduli space $\mathfrak{M}_{K^2,\chi}$ for every admissible pair $(K^2, \chi)$ such that $K^2=2\chi-5$. Furthermore, the surfaces constructed belong to connected components of $\overline{\mathfrak{M}}_{K^2,\chi}$ without canonical models.
\end{abstract}

\section{Introduction.}

Let $X$ be an algebraic surface over the complex numbers $\mathbb{C}$, which will be the ground field throughout the paper. The main numerical invariants of $X$ are the self-intersection of its canonical class $K^2_X$ and its holomorphic Euler characteristic $\chi(\mathcal{O}_X)$. 
 If $X$ is minimal and of general type it is well known (cf. \cite[Chapter VII]{Barth2004})
that the following inequalities are satisfied 
\begin{equation}\label{GeneralType}
 \chi(\mathcal{O}_X)\geq 1,\quad K_X^2\geq 1, \quad 2\chi(\mathcal{O}_X)-6\leq K_X^2\leq 9\chi(\mathcal{O}_X).
\end{equation}
Minimal algebraic surfaces of general type $X$   such that $K^2_X=2\chi(\mathcal{O}_X)-6$ were studied by Enriques \cite{EnriquesPaper}, \cite[Section VIII.11]{EnriquesBook} but they are frequently called Horikawa surfaces because of 
Horikawa's contribution to their deformation theory \cite{Hor1}. In particular,
denoting by $\mathfrak{M}_{K^2,\chi}$ Gieseker’s moduli space of canonical models of surfaces of general type
with fixed self-intersection of the canonical class $K^2$ and fixed holomorphic Euler characteristic $\chi$, Horikawa showed that (see Theorem \ref{DefClass}):
\begin{enumerate}
 \item[-] If $K^2=2\chi-6\notin 8\cdot\mathbb{Z}$ then $\mathfrak{M}_{K^2,\chi}$ has a unique connected component.
 \item[-] If $K^2=2\chi-6\in 8\cdot\mathbb{Z}$ then $\mathfrak{M}_{K^2,\chi}$ has two connected components $\mathfrak{M}^{I}_{K^2,\chi}$ and $\mathfrak{M}^{II}_{K^2,\chi}$.
\end{enumerate}

If $(K^2, \chi)$ is an admissible pair (i.e. a pair of integers satisfying the inequalities (\ref{GeneralType})) such that $K^2=2\chi-6$, 
every connected component of
Gieseker's moduli space $\mathfrak{M}_{K^2,\chi}$ contains surfaces admitting a $\mathbb{Z}_2^2$-action by \cite[Theorem 1]{Lorenzo2021}.
In this note $\mathbb{Z}_3$-actions on Horikawa surfaces are studied.
The main result is the following:

\begin{theorem}\label{Z3OnEvenHorikawa}
Let $(K^2, \chi)$ be an admissible pair such that $K^2=2\chi-6$. Then
every connected component of $\mathfrak{M}_{K^2, \chi}$ contains surfaces admitting a $\mathbb{Z}_3$-action.
\end{theorem}

That being said, let us denote by $\overline{\mathfrak{M}}_{K^2, \chi}$ the KSBA-compactification of Giekeser's moduli space $\mathfrak{M}_{K^2, \chi}$. Rollenske \cite{Rollenske2021} has recently proved  that for any admissible pair $(K^2, \chi)$ there exist non-normal stable surfaces $X$ with invariants $K^2_X=K^2$ and $\chi(\mathcal{O}_X)=\chi$ that do not admit a $\mathbb{Q}$-Gorenstein smoothing. Furthermore, the surfaces constructed by Rollenske belong to 
a connected component of $\overline{\mathfrak{M}}_{K^2, \chi}$  without canonical models.
The examples that will be considered in the proof of Theorem \ref{Z3OnEvenHorikawa} provide us with a method to construct normal surfaces with this property. We will illustrate the method by showing:

\begin{theorem}\label{MoreIrred}
Let $(K^2, \chi)$ be an admissible pair such that $K^2=2\chi-5$.
Then there exist normal stable surfaces $X$
with invariants $K^2_X=K^2$ and $\chi(\mathcal{O}_X)=\chi$ that belong to 
a connected component of $\overline{\mathfrak{M}}_{K^2, \chi}$  without canonical models.
\end{theorem}

The note is structured as follows. In Section \ref{SectionCyclicCovers}
how to construct cyclic abelian covers and to obtain information about them is explained. Section \ref{Horikawa surfaces.} is devoted to present some properties 
of minimal surfaces of general type $X$ with $K_X^2=2\chi(\mathcal{O}_X)-6$ and to study
their deformation equivalence classes. In Section \ref{FeDescr}
we present a description of the Hirzebruch surface $\mathbb{F}_e$ convenient for our
purposes.
In Section \ref{SectionZ3OnEvenHorikawa}
we prove Theorem
 \ref{Z3OnEvenHorikawa}.
Finally, Section \ref{Z3OnStable} contains a proof of Theorem \ref{MoreIrred} that takes into account the constructions of the previous section. 

 \section{Cyclic abelian covers.}\label{SectionCyclicCovers}
 Let $G$ be a finite abelian group. A $G$-cover of a variety $Y$ is a finite map $f\colon X\to Y$ together with a faithful action of $G$ on $X$ such that $f$ exhibits $Y$ as $X/G$.
The general case was first considered by Pardini \cite{Par1991}, but 
in this note we are just going to deal with $\mathbb{Z}_2$-covers and $\mathbb{Z}_3$-covers.
Abelian $G$-covers of surfaces with cyclic $G$ were already considered by Comessatti \cite{Comessatti}. Another reference for the particular case $G=\mathbb{Z}_2$ is \cite{PerssonDC}. Other references for the particular case $G=\mathbb{Z}_3$ are \cite{Mir1985} or \cite{Tan1991}.
 
 According to \cite[Proposition 2.1]{Par1991} (see also \cite[Remark 3.11]{MendesPardini2021}),
to define a $\mathbb{Z}_d$-cover $X\to Y$ of a smooth and irreducible projective variety $Y$ with normal $X$, it suffices to consider both:

\begin{enumerate}
 \item[-] Effective divisors $D_1,\ldots, D_{d-1}$ such that the branch locus $D_1+\cdots+D_{d-1}$ is reduced.
 \item[-] A line bundle $L$ satisfying $dL\equiv\sum_{j=1}^{d-1} j\cdot D_j$.
   \end{enumerate}
   The set $\{L,D_1,\ldots, D_{d-1}\}$ is called the reduced building data of the cover. 

\begin{remark}\label{RedRedBuildingData} 
 Note that if the Picard group of $Y$ has no $d$-torsion then the line bundle $L$ can be deduced from the divisors $D_1,\ldots, D_{d-1}$. 
 In this note we are only going to consider covers of simply connected surfaces, for which the Picard group has no torsion (cf. \cite[Remark 3.10]{MendesPardini2021}).
\end{remark}

For the reader's convenience we include the standard formulas for $\mathbb{Z}_2$-covers and $\mathbb{Z}_3$-covers in terms of the reduced building data.

\begin{proposition}
 [{{\cite[Section II]{PerssonDC}}}, 
 {{\cite[Proposition 4.2]{Par1991}}}
 ]\label{InvariantsZ2}
Let $Y$ be a smooth and irreducible projective surface and $f\colon X\to Y$ a smooth $\mathbb{Z}_2$-cover with reduced building data $\{L,D\}$. Then:
\begin{equation*}
\begin{split}
K_X\equiv f^*(K_Y+ L),\\
K_X^2=2(K_Y+ L)^2,\\
p_g(X)=p_g(Y)+h^0(K_Y+L),\\
\chi(\mathcal{O}_X)=2\chi(\mathcal{O}_Y)+\frac{1}{2}L(K_Y+L).
\end{split}
\end{equation*}
\end{proposition}

\begin{remark}\label{CanonicalImageSimpleCyclicCover}
Let $Y$ be a smooth surface and let us consider a smooth $\mathbb{Z}_2$-cover $\pi\colon X\to Y$ with reduced building data $\{L,D\}$. Let us assume that 
$h^0(K_Y)=0$. Then $K_X=\pi^*(K_Y+L)$ and $p_g(X)=h^0(K_Y+L)=:N$ by 
the standard formulas for $\mathbb{Z}_2$-covers (see Proposition \ref{InvariantsZ2}). If we denote  by $i\colon Y\dashrightarrow \mathbb{P}^{N-1}$ the (maybe rational) map defined by the complete linear system $|K_Y+L|$, it follows 
that $i\circ \pi$ is the map induced by the complete linear system $|K_X|$, i.e., it is the canonical map of $X$. In particular $i(Y)$ is the canonical image of $X$.
\end{remark}

\begin{proposition}
 [{{\cite[Proposition 10.3]{Mir1985}}}, 
 {{\cite[Proposition 4.2]{Par1991}}},
 {{\cite[Lemma 3.1]{Tan1991}}}
 ]\label{InvariantsCyclic}
Let $Y$ be a smooth and irreducible projective surface and $f\colon X\to Y$ a smooth $\mathbb{Z}_3$-cover with reduced building data $\{L,D_1,D_2\}$. Then:
\begin{equation*}
\begin{split}
3K_X\equiv f^*(3K_Y+ 2D_1+2D_2),\\
3K_X^2=(3K_Y+ 2D_1+2D_2)^2,\\
p_g(X)=p_g(Y)+h^0(K_Y+L)+h^0(K_Y+D_1+D_2-L),\\
\chi(\mathcal{O}_X)=3\chi(\mathcal{O}_Y)+\frac{1}{2}L(K_Y+L)+\frac{1}{2}
(D_1+D_2-L)(K_Y+D_1+D_2-L).
\end{split}
\end{equation*}
\end{proposition}

 \begin{example}\label{1311Singularities}
Let $r\colon S\to X$ be the minimal resolution of a normal surface singularity $(X,p)$
whose exceptional divisor $r^{-1}(p)$ is a $(-3)$-curve $C$. Then $p$ is said to be a $\frac{1}{3}(1,1)$-singularity.
 These singularities are log canonical because $K_S=r^*K_X-\frac{1}{3}C$ 
 (cf. \cite{Alexeev1992}). 
 
 There is an easy way to obtain 
 $\frac{1}{3}(1,1)$-singularities via $\mathbb{Z}_3$-covers (cf. \cite{Tan1991}). Indeed, 
 let $\pi\colon X\to Y$ be a $\mathbb{Z}_3$-cover with reduced building data $\{L,D_1,D_2\}$.
 Let us assume that $D_1$ and $D_2$ intersect in a point $q\in Y$ giving rise to an ordinary double point
 on the branch locus $B=D_1+D_2$ of $\pi$. Then $X$ has a $\frac{1}{3}(1,1)$-singularity $p$ over $q$ (see \cite[Proposition 3.3]{Par1991}).
 We can resolve this singularity in a canonical way 
 (see \cite[Section II]{Tan1991}).
 Let $b\colon \widetilde{Y}\to Y$ be the blow-up of $Y$ at $q$ with exceptional divisor $E$. Then there is 
 a $\mathbb{Z}_3$-cover $\widetilde{X}\to\widetilde{Y}$ with branch locus $\widetilde{B}=\widetilde{D_1}+\widetilde{D_2}$
 where $\widetilde{D_i}=b^*D_i-E$. The induced map $\widetilde{X}\to X$ resolves the 
$\frac{1}{3}(1,1)$-singularity. Moreover, using the formulas for $\mathbb{Z}_3$-covers it can be proved that 
$\chi(\mathcal{O}_X)=\chi(\mathcal{O}_{\widetilde{X}})$ and $K^2_X=K^2_{\widetilde{X}}+\frac{1}{3}$.

We note that $\frac{1}{3}(1,1)$-singularities do not admit $\mathbb{Q}$-Gorenstein smoothings 
by \cite[Proposition 3.11]{KSB88}. 
Moreover, they are $\mathbb{Q}$-Gorenstein rigid (cf. \cite[Basic Concepts]{AkhtarEtAl2016}).
\end{example}

\section{Horikawa surfaces on the line \texorpdfstring{$K^2=2\chi-6$}{K2chi6}.} \label{Horikawa surfaces.} 

Horikawa \cite{Hor1} studied minimal surfaces of general type $X$ such that $K_X^2=2\chi(\mathcal{O}_X)-6$.
The following theorems are some of the results proved in \cite{Hor1}.

\begin{theorem}[{{\cite[Lemma 1.1]{Hor1}}}]\label{StructureEvenHorikawa}
Let $X$ be a minimal algebraic surface with $K_X^2=2\chi(\mathcal{O}_X)-6$
and $\chi(\mathcal{O}_X) \geq 4$. Then the canonical system $|K_X|$ has no base point. Moreover,
the canonical map $\varphi_{K_X}\colon X\to \mathbb{P}^{p_g(X)-1}$ is a morphism 
of degree 2 onto a surface of degree $p_g(X)-2$ in $\mathbb{P}^{p_g(X)-1}$.
\end{theorem}

\begin{theorem}[{{\cite[Theorem 3.3, Theorem 4.1 and Theorem 7.1]{Hor1}}}]
\label{DefClass}
Let $(K^2, \chi)$ be an admissible pair such that $K^2=2\chi-6$.
If $K^2\notin 8\cdot\mathbb{Z}$ then minimal algebraic surfaces $X$ such
that $K^2_X=K^2$ and $\chi(\mathcal{O}_X)=\chi$ have one and the same deformation type.
If $K^2\in 8\cdot\mathbb{Z}$ then minimal algebraic surfaces $X$ such 
that $K^2_X=K^2$ and $\chi(\mathcal{O}_X)=\chi$ have two deformation classes.
The image of the canonical map of a surface in the first class is $\mathbb{F}_e$
for some $e\in\{0,2,\ldots,\frac{1}{4}K^2\}$. The image of the canonical map of a 
surface in the second class is $\mathbb{F}_{\frac{1}{4}K^2+2}$ if $K^2>8$ and $\mathbb{P}^2$
or a cone over a rational curve of degree $4$ in $\mathbb{P}^4$ if $K^2=8$.
\end{theorem}

\begin{remark}
Given $k\geq 1$ we will denote by $\mathfrak{M}_{8k, 4k+3}^{I}$ (resp. $\mathfrak{M}_{8k, 4k+3}^{II}$) 
the connected component of $\mathfrak{M}_{8k, 4k+3}$ containing  the surfaces on the first
(resp. second) deformation class.
\end{remark}

\begin{remark}\label{M1vsM2}
A smooth surface $Y$ is a deformation of the Hirzebruch surface $\mathbb{F}_m$ if and only if $Y$ is isomorphic to the Hirzebruch surface $\mathbb{F}_n$ for some $n$ such that $n\equiv m(2)$ (cf. \cite[Theorem VI.8.iv]{Barth2004}).
On the other hand, cones over a rational curve of degree $4$ in $\mathbb{P}^4$ are degenerations of $\mathbb{P}^2$ (cf. \cite{Manetti1991}). This gives us an insight about why:
\begin{enumerate}
 \item[(i)] surfaces in $\mathfrak{M}_{8, 7}^{I}$, which are $\mathbb{Z}_2$-covers of the Hirzebruch surface $\mathbb{F}_0$ or the Hirzebruch surface $\mathbb{F}_2$, are deformation equivalent;
 \item[(ii)] surfaces in $\mathfrak{M}_{8, 7}^{II}$, which are $\mathbb{Z}_2$-covers of $\mathbb{P}^2$ or a cone over a rational curve of degree $4$ in $\mathbb{P}^4$, are deformation equivalent;
 \item[(iii)] surfaces in $\mathfrak{M}_{8, 7}^{I}$ are not deformation equivalent to surfaces in $\mathfrak{M}_{8, 7}^{II}$.
\end{enumerate}
In the case $k\geq2$ one may wonder why surfaces in $\mathfrak{M}_{8k, 4k+3}^{I}$ are not deformation equivalent to surfaces in $\mathfrak{M}_{8k, 4k+3}^{II}$ if they are all $\mathbb{Z}_2$-covers of deformation equivalent Hirzebruch surfaces. The reason is that surfaces in $\mathfrak{M}_{8k, 4k+3}^{I}$ are $\mathbb{Z}_2$-covers with connected branch locus whereas 
surfaces in $\mathfrak{M}_{8k, 4k+3}^{II}$ are $\mathbb{Z}_2$-covers with disconnected branch locus.
\end{remark}

\begin{remark}\label{CanonicalHorikawaInM2}
 As we saw in Remark \ref{M1vsM2}, in the case $k\geq2$ the canonical map of a surface $X\in\mathfrak{M}^{II}_{8k, 4k+3}$
induces a  $\mathbb{Z}_2$-cover 
 $\varphi\colon X\to \mathbb{F}_{2k+2}$ of the Hirzebruch surface $\mathbb{F}_{2k+2}$ with negative section $\Delta_0$ of
 self-intersection 
 $-(2k+2)$ and fiber $F$ whose branch locus $B$ is disconnected. More precisely, $B$ consists of $\Delta_0$ and a divisor
 $B'\in|5\Delta_0+10(k+1)F|$ having at 
 most canonical singularities. 
 In particular, it follows from the Riemann-Hurwitz formula that $|\varphi^*F|$ induces a genus $2$ fibration on $X$.
 If we denote by $H$ a general genus $2$ fiber of $X$
 and $\Gamma=(\varphi^*\Delta_0)_{\text{red}}$, 
 the canonical class of $X$ is $2\Gamma+(3k+1)H$ by 
 the standard formulas for $\mathbb{Z}_2$-covers 
 (see Proposition \ref{InvariantsZ2}).
 Thus, the self-intersection of an irreducible component $C$ of a genus $2$ fiber of $X$ has to be even because:
 \begin{equation*}
  C^2=2p_a(C)-2-CK_X=2p_a(C)-2-2C\Gamma.
 \end{equation*} 
 \end{remark}
 
\section{A description of the Hirzebruch surface \texorpdfstring{$\mathbb{F}_e$}{Fe}.}\label{FeDescr}

This section is devoted to present a description of the Hirzebruch surface $\mathbb{F}_e$ with negative section $\Delta_0$ of self-intersection $(-e)$ and fiber $F$ convenient for our purposes. Most of what is gathered in this section can be found in \cite[Chapter 2]{ReidChapters}, but we include it to fix the notation.

Denote $\mathbb{C}^*=\mathbb{C}\setminus\{0\}$. Reid \cite[Chapter 2]{ReidChapters} defined the rational scroll $\mathbb{F}(0,e)\simeq \mathbb{F}_e$ as the quotient of $(\mathbb{A}^2\setminus \{0\})\times (\mathbb{A}^2\setminus \{0\})$ under the action of the group $\mathbb{C}^*\times\mathbb{C}^*$ given by:
\begin{center}
 $\xymatrix@R-2pc{
(\mathbb{C}^*\times\mathbb{C}^*)\times ((\mathbb{A}^2\setminus \{0\})\times (\mathbb{A}^2\setminus \{0\}))\ar[r]&(\mathbb{A}^2\setminus \{0\})\times (\mathbb{A}^2\setminus \{0\}),\\
  ((\lambda,\mu), (t_1,t_2;x_1,x_2)) \ar@{|-{>}}[r]& (\lambda t_1,\lambda t_2;\mu x_1, \frac{\mu}{\lambda^{e}} x_2).
}$
\end{center}
The class of $(t_1,t_2;x_1,x_2)\in (\mathbb{A}^2\setminus \{0\})\times (\mathbb{A}^2\setminus \{0\})$ under the equivalence relation induced by this action  will be denoted by $(t_1\colon t_2; x_1\colon x_2)\in\mathbb{F}_e$.
This description allows to define curves on $\mathbb{F}_e$ via polynomials in $\mathbb{C}[t_1,t_2,x_1,x_2]$ whose set of zeros is invariant by the 
action of $\mathbb{C}^*\times\mathbb{C}^*$. In particular, if $e\neq 0$:
\begin{enumerate}
 \item[-] $(\alpha t_1+\beta t_2=0)\in|F|$ defines a different fiber 
 for each $(\alpha\colon \beta)\in\mathbb{P}^1$;
 \item[-] $(x_2=0)\in|\Delta_0|$ defines the negative section;
 \item[-] $(x_1=0)\in|\Delta_0+eF|$ defines an irreducible section disjoint from the negative section.
\end{enumerate}
In addition, $\mathbb{F}_e$ can be covered by the open subsets $(t_ix_j\neq 0), i,j\in\{1,2\}$ and each open subset $(t_ix_j\neq 0)$ is isomorphic to $\mathbb{A}^2$. 
 
 \section{Horikawa surfaces with a \texorpdfstring{$\mathbb{Z}_3$}{Z3}-action.}\label{SectionZ3OnEvenHorikawa}
 
The aim of this section is to prove Theorem \ref{Z3OnEvenHorikawa}. 
We are going to proceed as follows. 
 We fix an admissible pair $(K^2, \chi)$ such that $K^2=2\chi-6$.
To begin with, we are going to find a surface in $\mathfrak{M}_{K^2, \chi}$ with a $\mathbb{Z}_3$-action. Then
 we are going to check that it
belongs to the first deformation class when $K^2\in 8\cdot \mathbb{Z}$. Finally, we are going to find a surface with a $\mathbb{Z}_3$-action in 
$\mathfrak{M}^{II}_{8k, 4k+3}$ for every integer $k\geq 1$.

%

 Let us choose integers $e,\alpha,\beta$ as follows:
\begin{enumerate}
 \item[-] If $\chi\equiv 0(3)$ we take $e=1, \alpha=\chi,\beta=3$.
 \item[-] If $\chi\equiv 1(3)$ we take $e=0, \alpha=\chi,\beta=1$.
 \item[-] If $\chi\equiv 2(3)$ we take $e=2, \alpha=\chi,\beta=5$.
\end{enumerate}
Let $\mathbb{F}_e$ be the Hirzebruch surface with negative section $\Delta_0$ of self-intersection $(-e)$
and fiber $F$. We consider smooth and irreducible divisors $D_1\in |2\Delta_0+\alpha F|$ and 
$D_2\in|2\Delta_0+\beta F|$ intersecting transversally in $2\alpha+2\beta-4e$ points in general position
$p_1,\ldots,p_{2\alpha+2\beta-4e}$. We are going to construct the canonical resolution (see Example \ref{1311Singularities}) of a 
$\mathbb{Z}_3$-cover of $\mathbb{F}_e$ with branch locus $B=D_1+D_2$.
Let $q\colon \widetilde{\mathbb{F}}_e\to\mathbb{F}_e$ 
be the blow-up of $\mathbb{F}_e$ at $p_1,\ldots,p_{2\alpha+2\beta-4e}$ with exceptional divisors 
$E_1,\ldots,E_{2\alpha+2\beta-4e}$. We define a smooth $\mathbb{Z}_3$-cover $\pi\colon S\to \widetilde{\mathbb{F}}_e$
with branch locus $\widetilde{B}=\widetilde{D}_1+\widetilde{D}_2$ consisting of
\begin{equation*}
 \begin{split}
  \widetilde{D}_1=q^*D_1-\sum_{i=1}^{2\alpha+2\beta-4e}E_i,\\
  \widetilde{D}_2=q^*D_2-\sum_{i=1}^{2\alpha+2\beta-4e}E_i.
 \end{split}
\end{equation*}
It follows from Proposition \ref{InvariantsCyclic} that:
\begin{equation*}
 \begin{split}
 3K_S\equiv\pi^*(3K_{\widetilde{\mathbb{F}}_e}+2\widetilde{D}_1+2\widetilde{D}_2)\equiv \pi^*\left((\alpha+2\beta-3e-6)q^*F+\widetilde{D}_1\right),\\
  K^2_S=2\alpha+2\beta-4e-8=K^2,\\
  \chi(\mathcal{O}_S)=\alpha+\beta-2e-1=\chi.
 \end{split}
\end{equation*}
Since  $3K_{S}$ is the pullback via $\pi$ of a nef divisor, $S$ is minimal.
Therefore the canonical model of $S$ belongs
to $\mathfrak{M}_{K^2, \chi}$ and has a $\mathbb{Z}_3$-action. 

Now we are going to show that the surfaces that we have just constructed belong to $\mathfrak{M}^I_{K^2, \chi}$ when $\chi=4k+3$ and $K^2=8k$ for some integer $k\geq 1$. 

In the case $k\geq2$ we notice that the fibration $\mathbb{F}_e\to \mathbb{P}^1$ induces a genus $2$ fibration on $S$ such that the genus $2$ fiber of $S$ 
 corresponding to the fiber of $\mathbb{F}_e$ through 
 $p_i,i\in\{1,\ldots,2\alpha+2\beta-4e\}$ consists of two $(-3)$-curves 
 intersecting transversally in three different points (see Figure \ref{PictureFibers2x(-3)}). Moreover, $S$ does not admit another genus $2$ fibration by \cite[Proposition 6.4]{XiaoGangGen2}.
 Since the irreducible components of the fibers of the genus $2$ fibration that a surface in $\mathfrak{M}^{II}_{8k, 4k+3}$ has by Remark \ref{CanonicalHorikawaInM2} have even self-intersection (see Remark \ref{CanonicalHorikawaInM2}), we conclude that
 the canonical model of $S$ belongs to $\mathfrak{M}^{I}_{K^2, \chi}$ as claimed.

 In the case $k=1$ the divisor $\widetilde{D}_2\subset \widetilde{\mathbb{F}}_e$ is a $2$-section with self-intersection $-12$. 
 Hence $\pi^*(\widetilde{D}_2)_{\text{red}}$ is a $2$-section of $S$ with self-intersection $-4$. I claim that its image via the canonical 
 map of $S$ 
 is a section with self-intersection $-2$. Indeed, the divisor $\pi^*(\widetilde{D}_1+\widetilde{D}_2)_{\text{red}}$ is invariant by the canonical involution $\tau$ of $S$ and the divisors $\pi^*(\widetilde{D}_1)_{\text{red}}$ and $\pi^*(\widetilde{D}_2)_{\text{red}}$ cannot be switched by it since they have different linear equivalence classes. Then $\pi^*(\widetilde{D}_2)_{\text{red}}$ is invariant by $\tau$ and we have two possibilities depending on whether $\tau$ restricted to $\pi^*(\widetilde{D}_2)_{\text{red}}$ is the trivial automorphism or not.
 In the former case the canonical map of $S$ sends 
 $\pi^*(\widetilde{D}_2)_{\text{red}}$ to a self-intersection $-8$ bisection of the canonical image of $S$, which contradicts Theorem \ref{DefClass}.
  Therefore $\tau$ restricted to $\pi^*(\widetilde{D}_2)_{\text{red}}$ is not the trivial automorphism and
 the canonical map of $S$ sends 
 $\pi^*(\widetilde{D}_2)_{\text{red}}$ to a self-intersection $-2$ section of the canonical image of $S$ as claimed. 
 It follows from Theorem \ref{DefClass} that the canonical map of $S$ has $\mathbb{F}_2$ as image and sends 
 $\pi^*(\widetilde{D}_2)_{\text{red}}$ to the negative section of this Hirzebruch surface. 
 Therefore the canonical model of $S$ belongs to $\mathfrak{M}^{I}_{8,7}$ again by Theorem \ref{DefClass}.
 
  \begin{figure}
\begin{center}
\scalebox{0.8}{
 \begin{tikzpicture}
 \draw[thick, dashed](0,2)--(0,-2) ;
  \draw[thick, blue](-1/2,1/2)--(1/2,3/2);
   \filldraw(0,1) circle (2pt) node at (-3/4,3/4) {$p_i$};
 \draw[thick](-1/2,3/2)--(1/2,1/2) node at (-1,2) {$\mathbb{F}_{e}$} ;
 \draw[thick](-1/2,-1/2)--(1/2,-1/2) node at (-1,-1/2) {$D_1$} ;
 \draw[thick, blue](-1/2,-1)--(1/2,-1)  node at (-1,-1) {$D_2$};
  \draw[->] (3,0)--(1,0) node at (2,1/2) {$q$};
 \end{tikzpicture}
 \begin{tikzpicture}
 \draw[thick, dashed](0,2)--(0,-2);
 \draw[thick, dashed](-1/2,1)--(3/2,1) node at (-3/4,3/4) {$E_i$} ;
 \draw[thick](1/2,1/2)--(1/2,3/2) node at (-1,2) {$\widetilde{\mathbb{F}}_e$} ;
 \draw[thick, blue](1,1/2)--(1,3/2);
 \draw[thick](-1/2,-1/2)--(1/2,-1/2) node at (-1,-1/2) {$\widetilde{D}_1$};
 \draw[thick, blue](-1/2,-1)--(1/2,-1) node at (-1,-1) {$\widetilde{D}_2$};
  \draw[->] (4,0)--(2,0) node at (3,1/2) {$\pi$};
 \end{tikzpicture}
  \begin{tikzpicture}
 \draw[thick, white](0,2)--(0,-2);
 \draw[scale=0.5, thick, smooth, domain=-3:3, variable=\x] plot ({sin(\x*80)}, \x) node at (-2,3) {$S$};
 \draw[scale=0.5, thick, smooth, domain=-3:3, variable=\x] plot ({-sin(\x*80)}, \x);
 \end{tikzpicture}
 }
 \end{center}
 \caption{Fiber of $S$ corresponding to the fiber of $\mathbb{F}_e$ through $p_i$.} \label{PictureFibers2x(-3)}
  \end{figure}
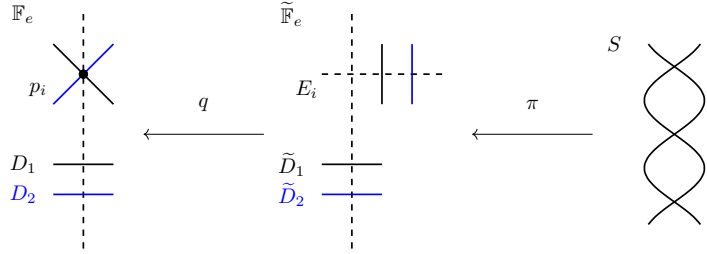

\medskip

Now we are going to construct surfaces with a $\mathbb{Z}_3$-action in $\mathfrak{M}_{8k, 4k+3}^{II}$ for every integer $k\geq 1$.

Let us assume first that $k=1$. Then we consider a $\mathbb{Z}_2$-cover $X\to \mathbb{P}^2$ of 
 $\mathbb{P}^2=\text{Proj}(\mathbb{C}[X_0,X_1,X_2])$ branched along
 the smooth and irreducible curve $B=(X_0^{10}+X_1^{10}+X_2^{10}=0)\in|\mathcal{O}_{\mathbb{P}^2}(10)|$. We also consider the order $3$
 automorphism  
 \begin{equation*}
  \sigma\colon\mathbb{P}^2\to \mathbb{P}^2, (X_0\colon X_1\colon X_2)\mapsto (X_2\colon X_0 \colon X_1).
 \end{equation*}
 Since $B$ is invariant under $\sigma$, the automorphism $\sigma$ lifts to an automorphism of $X$ by \cite[Section 2.2]{MendespardiniLift}. It follows
 from Remark \ref{CanonicalImageSimpleCyclicCover} that the canonical image of $X$ is $\mathbb{P}^2$ and therefore
 Theorem \ref{DefClass} allows us to conclude that
 $X$ is a smooth Horikawa surface  in
 $\mathfrak{M}^{II}_{8,7}$ that has a $\mathbb{Z}_3$-action.
 
 Now we assume $k\geq2$. Using the notation of Section \ref{FeDescr}, we define a curve $C\in |5\Delta_0+(10k+10)F|\subset\mathbb{F}_{2k+2}$ as follows: 
 \begin{enumerate}
  \item[-] If $k\equiv 2(3)$, $C=(x_1^5+x_2^5t_1^{10k+10}+x_2^5t_2^{10k+10}=0)$.
  \item[-] If $k\equiv 0(3)$, $C=(x_1^5+x_2^5t_1^{10k+9}t_2+x_2^5t_2^{10k+10}=0)$.
  \item[-] If $k\equiv 1(3)$, $C=(x_1^5+x_2^5t_1^{10k+8} t_2^2+x_2^5t_2^{10k+10}=0)$.
 \end{enumerate}
 Writing the equation of $C$ on $(t_ix_j\neq 0)\simeq \mathbb{A}^2, i,j\in\{1,2\}$, a straightforward computation yields that in the first two cases $C$ is smooth and in the third case $C$ has a unique singularity at $(1\colon 0;0\colon 1)$. This singularity is isomorphic to $(C',0)$ where:
 \begin{equation*}
  C'=\text{Spec}\left(\frac{\mathbb{C}[a,b]}{(a^2+a^{10k+10}+b^5)}\right)
 \end{equation*}
and therefore it is of type $A_4$ (see \cite[Section II.8]{Barth2004}).
In particular, if we denote by $\pi\colon X\to \mathbb{F}_{2k+2}$ the $\mathbb{Z}_2$-cover of $\mathbb{F}_{2k+2}$ branched along $B=\Delta_0+C$, either the surface $X$ is smooth or it has a singularity of type $A_4$.
Moreover, by the standard formulas for $\mathbb{Z}_2$-covers:
\begin{equation*}
 \begin{split}
  2K_X\equiv\pi^*(2K_{\mathbb{F}_{2k+2}}+B)\equiv\pi^*(2\Delta_0+(6k+2)F),\\
  K^2_X=8k,\\
  \chi(\mathcal{O}_X)=4k+3.
 \end{split}
\end{equation*}
Note that $K_X$ is ample because it is the pullback via $\pi$ of the ample divisor $\Delta_0+(3k+1)F$. Therefore $X$ is a canonical model. In addition, 
since $\Delta_0+(3k+1)F$ is not only ample but very ample and $h^0(K_{\mathbb{F}_{2k+2}})=0$ it follows from Remark 
\ref{CanonicalImageSimpleCyclicCover} that the canonical map 
of $X$ is the composition 
of $\pi$ with the map induced by the complete linear system $|\Delta_0+(3k+1)F|$ and therefore the canonical 
image of $X$ is $\mathbb{F}_{2k+2}$. Therefore $X$ belongs to $\mathfrak{M}^{II}_{8k, 4k+3}$ by Theorem \ref{DefClass}.

Let us consider the order $3$ automorphism 
\begin{equation*}
 \sigma\colon\mathbb{F}_{2k+2}\to \mathbb{F}_{2k+2}, (t_1\colon t_2;x_1\colon x_2)
 \mapsto (\zeta t_1\colon t_2;x_1\colon x_2),
\end{equation*}
where $\zeta$ is a primitive $3$-root of unity. Since $C$ and $\Delta_0$ are invariant by $\sigma$, the automorphism $\sigma$ induces an order $3$ automorphism of $X$
by \cite[Section 2.2]{MendespardiniLift}. Thus $X\in \mathfrak{M}_{8k, 4k+3}^{II}$ has a $\mathbb{Z}_3$-action. This finishes the proof of Theorem \ref{Z3OnEvenHorikawa}.

\section{Contracting \texorpdfstring{$(-3)$}{3}-curves.}\label{Z3OnStable}
Let $(K^2, \chi)$ be an admissible pair such that $K^2=2\chi-6$. In Section \ref{SectionZ3OnEvenHorikawa} we were able to construct
surfaces $S\in \mathfrak{M}_{K^2, \chi}$ with a $\mathbb{Z}_3$-action. Moreover, we constructed surfaces $S$ that had a genus 
$2$ fibration with $2\chi+2$
fibers consisting of two $(-3)$-curves intersecting transversally in three different points (see Figure \ref{PictureFibers2x(-3)}).
In particular, we can choose three disjoint $(-3)$-curves of $S$ and contract them $c\colon S\to X$. 
According to Example \ref{1311Singularities} the surface 
$X$ has three $\frac{1}{3}(1,1)$-singularities and
\begin{equation*}
 K^2_X=K^2_S+1=2\chi(\mathcal{O}_S)-5=2\chi(\mathcal{O}_X)-5.
\end{equation*}
The proof of Theorem \ref{MoreIrred} consists in constructing carefully the surfaces $X$ just described.

\begin{proof}[Proof of Theorem \ref{MoreIrred}]
 Given an admissible pair $(K^2, \chi)$ such that $K^2=2\chi-5$ we choose integers $e,\alpha,\beta$ as follows:
\begin{enumerate}
 \item[-] If $\chi\equiv 0(3)$ we take $e=1, \alpha=\chi,\beta=3$.
 \item[-] If $\chi\equiv 1(3)$ we take $e=0, \alpha=\chi,\beta=1$.
 \item[-] If $\chi\equiv 2(3)$ we take $e=2, \alpha=\chi,\beta=5$.
\end{enumerate}
Let $\mathbb{F}_e$ be the Hirzebruch surface with negative section $\Delta_0$ of self-intersection $(-e)$ and fiber $F$.
We consider smooth and irreducible divisors $D_1\in |2\Delta_0+\alpha F|$ and 
$D_2\in|2\Delta_0+\beta F|$ intersecting transversally in $2\alpha+2\beta-4e$ points 
in general position $p_1,\ldots,p_{2\alpha+2\beta-4e}$. Let $q\colon \overline{\mathbb{F}}_e\to\mathbb{F}_e$ 
be the blow-up of $\mathbb{F}_e$ at $p_4,\ldots,p_{2\alpha+2\beta-4e}$ with exceptional divisors 
$E_4,\ldots,E_{2\alpha+2\beta-4e}$. We define a $\mathbb{Z}_3$-cover $\pi\colon X\to \overline{\mathbb{F}}_e$
with branch locus $\overline{B}=\overline{D}_1+\overline{D}_2$ consisting of
\begin{equation*}
 \begin{split}
  \overline{D}_1=q^*D_1-\sum_{i=4}^{2\alpha+2\beta-4e}E_i,\\
  \overline{D}_2=q^*D_2-\sum_{i=4}^{2\alpha+2\beta-4e}E_i.
 \end{split}
\end{equation*}
It follows from the formulas for $\mathbb{Z}_3$-covers that:
\begin{equation*}
\begin{split}
 3K_X=\pi^*\left(q^*(2\Delta_0+(2\alpha+2\beta-3e-6)F)-\sum_{i=4}^{2\alpha+2\beta-4e}E_i\right),\\
  K^2_X=2\alpha+2\beta-4e-7=2\chi-5,\\
  \chi(\mathcal{O}_{X})=\alpha+\beta-2e-1=\chi.
\end{split}
\end{equation*}
In addition, using Nakai-Moishezon criterion we can prove that the divisor 
\begin{equation*}
 D:=q^*(2\Delta_0+(2\alpha+2\beta-3e-6)F)-\sum_{i=4}^{2\alpha+2\beta-4e}E_i
\end{equation*}
is ample. Indeed, let us suppose that there exists an irreducible curve $C\in |q^*(a\Delta_0+bF)-\sum_{i=4}^{2\alpha+2\beta-4e}c_iE_i|$ such that $CD<0$ for some non negative integers $a,b,c_i$. This implies that $q(C)\in |a\Delta_0+bF|$ is an irreducible curve of $\mathbb{F}_{e}$ such that the sum of the multiplicities 
 of $q(C)$ at $p_4,\ldots, p_{2\alpha+2\beta-4e}$ is
 \begin{equation*}
  \sum_{i=4}^{2\alpha+2\beta-4e}c_i>(2\alpha+2\beta-5e-6)a+2b.
 \end{equation*}
 Now, if we choose $R\in|\Delta_0+(\alpha+\beta-e-2)F|$ passing through $p_4,\ldots, p_{2\alpha+2\beta-4e}$ then:
 \begin{equation*}
   (\alpha+\beta-2e-2)a+b=R\cdot q(C)\geq \sum_{i=4}^{2\alpha+2\beta-4e}c_i>(2\alpha+2\beta-5e-6)a+2b.
 \end{equation*}
 By the choices of $\alpha,\beta, e$ this is only possible if $\alpha=\beta=3, e=1$,
 which corresponds to the case $(K^2,\chi)=(1,3)$.  Moreover, since $q(C)$ is irreducible, we necessarily have $a=1, b=0$ (cf. \cite[Corollary V.2.18]{Hartshorne1977}), i.e. $q(C)$ is the negative section of $\mathbb{F}_1$ and passes through $p_4,\ldots, p_8$. This contradicts the fact that the points $p_i$ are in general position.
Since $3K_X$ is the pullback of $D$ via $\pi$, we conclude that $K_X$ is ample.
On the other hand, the only singularities of $X$ are three $\frac{1}{3}(1,1)$-singularities over $\overline{D}_1\cap \overline{D}_2$.
Since these singularities are log canonical but do not admit a $\mathbb{Q}$-Gorenstein smoothing (see Example \ref{1311Singularities}), it follows that $X$ is a non-smoothable stable surface. In other words, $X$ belongs to the KSBA-compactification
$\overline{\mathfrak{M}}_{K^2, \chi}$ of Gieseker's moduli space $\mathfrak{M}_{K^2, \chi}$ but
it is contained in an irreducible component of $\overline{\mathfrak{M}}_{K^2, \chi}$ without canonical models. 

Furthermore, $X$ is contained in a connected component of $\overline{\mathfrak{M}}_{K^2, \chi}$ without canonical models. Indeed, according to \cite[Corollary 2.6]{Rollenske2021} it suffices to show that $h^0(2K_X)\neq \chi(\mathcal{O}_X)+ K^2_X$. Now, $h^0(2K_X)=\chi(2K_X)$ by \cite[Proposition 3.6]{LiuRollenske2014} and the right hand-side of this equality can be computed using the Riemann-Roch theorem for Weil divisors on normal surfaces \cite[Theorem 1.2]{Blache}. More precisely, denoting by $\text{Sing} (X)$ the singular locus of $X$,
\begin{equation*}
 \chi(2K_X)=\chi(\mathcal{O}_X)+K^2_X+\sum_{x\in \text{Sing} (X)} R_{X,x}(2K_X),
\end{equation*}
where the contribution $R_{X,x}(2K_X)$ is a local correction term that equals $-\frac{1}{3}$ if $x\in X$ is a $\frac{1}{3}(1,1)$-singularity (see \cite[Lemma 5.4]{Blache}). 
Since the only singularities of $X$ are three $\frac{1}{3}(1,1)$-singularities,
\begin{equation*}
 h^0(2K_X)=\chi(2K_X)=\chi(\mathcal{O}_X)+K^2_X-1\neq \chi(\mathcal{O}_X)+K^2_X
\end{equation*}
and our claim follows.
\end{proof}

\begin{remark}
Let $\varepsilon$ be a positive integer such that $3\cdot \varepsilon\leq 2\chi+2$. Then we can contract $3\cdot \varepsilon$ disjoint $(-3)$-curves of $S$. We obtain in this way a surface $X$ with $3\cdot \varepsilon$ singularities of type $\frac{1}{3}(1,1)$ and $K^2_X=2\chi(\mathcal{O}_X)-6+\varepsilon$. In particular $3K_X^2\leq 8\chi(\mathcal{O}_X)-16$.
\end{remark}

\noindent \begin{acknowledgements}
The author is deeply indebted to his supervisor Margarida Mendes Lopes for all her help. The author also thanks Rita Pardini for pointing out some mistakes and suggesting ways to improve the note. Thanks are also due to Sönke Rollenske for noticing that the stable surfaces constructed belong to a connected component without canonical models and not just to an irreducible component without canonical models. Finally, the author would like to express his gratitude to the anonymous reviewers for their thorough reading of the paper and suggestions.
\end{acknowledgements}

\bibliographystyle{plain}      
\bibliography{Z3ActionsOnHorikawaSurfaces}
 \vspace{5mm}

\noindent Vicente Lorenzo \footnote{The author is a doctoral student of the Department of Mathematics and
Center for Mathematical Analysis, Geometry and Dynamical Systems of Instituto Superior T\'{e}cnico,
Universidade de Lisboa and is supported by Fundac\~{a}o para a Ci\^{e}ncia e a Tecnologia (FCT), Portugal through
the program Lisbon Mathematics PhD (LisMath), scholarship  FCT - PD/BD/128421/2017 and
projects UID/MAT/04459/2019 and UIDB/04459/2020.}\\
Center for Mathematical Analysis, Geometry and Dynamical Systems\\
Departamento de Matem\'{a}tica\\
Instituto Superior T\'{e}cnico\\
Universidade de Lisboa\\
Av. Rovisco Pais\\
1049-001 Lisboa\\
Portugal\\
\textit{E-mail address: }{vicente.lorenzo@tecnico.ulisboa.pt}\\
https://orcid.org/0000-0003-2077-6095

\end{document}